\documentclass[11pt,oneside,reqno,a4paper]{amsart}
\usepackage[T1]{fontenc}
\usepackage[utf8]{inputenc}
\usepackage[english]{babel}
\usepackage{setspace}
\usepackage[margin=1in,headsep=0.4in,footskip=0.2in]{geometry}
\usepackage[dvipsnames]{xcolor}
\usepackage{amsmath,amssymb,amsthm}
\usepackage{amsfonts}
\usepackage{enumitem} 

\usepackage{siunitx}
 \usepackage{graphicx}
\usepackage{hyperref}
\hypersetup{
 colorlinks,
 citecolor=Maroon,
 linkcolor=Maroon,
 urlcolor=Maroon}
\usepackage{esint}
\usepackage{lmodern}
\usepackage{bbm}

\usepackage[
  rm={oldstyle=false,proportional=true},
  sf={oldstyle=true,proportional=true},
  tt={oldstyle=true,proportional=true,variable=true},
  qt=false,
]{cfr-lm}

\makeatletter
\def\@seccntformat#1{\hspace*{0mm}%
 \protect\textup{\protect\@secnumfont
   \ifnum\pdfstrcmp{subsection}{#1}=0 \bfseries\fi
   \csname the#1\endcsname
   \protect\@secnumpunct
     }%
}
\numberwithin{equation}{section}

\newtheorem{theorem}{Theorem}[section]
\newtheorem{corollary}[theorem]{Corollary}

\theoremstyle{definition}

\newtheorem{remark}{Remark}[section]

\newcommand{\assign}{:=}
\newcommand{\mathd}{\mathrm{\, d}}
\newcommand{\of}{:}
\newcommand{\suchthat}{:}
\newcommand{\tmabbr}[1]{#1}

\newcommand{\tmem}[1]{{\em #1\/}}

\newcommand{\tmop}[1]{\ensuremath{\operatorname{#1}}}

\newcommand{\pp}{p}
\newcommand{\qq}{q}
\newcommand{\Om}{\Omega}
\newcommand{\Ndim}{N}

\newcommand{\divv}{\mathrm{div}}
\newcommand{\grad}{\nabla}

\newcommand{\Stwo}{\mathbb{S}}
\newcommand{\RR}{\mathbb R}
\newcommand{\NN}{\mathbb N}
\newcommand{\eqs}{\hspace{0.1em}=\hspace{0.1em}}

\begin{document}

\title[A unified divergent approach to Hardy--Poincar{\'e} Inequalities]{A unified divergent approach to Hardy--Poincar{\'e} Inequalities
in classical and variable Sobolev Spaces}

\author{Giovanni Di Fratta}
\address{Giovanni Di Fratta\\ Institute for Analysis and Scientific Computing, TU Wien, Wiedner
Hauptstrae 8-10, 1040 Wien, Austria}
\email{giovanni.difratta@asc.tuwien.ac.at}

\author{Alberto Fiorenza}
\address{Alberto Fiorenza \\
Dipartimento di Architettura \\
Universit\`{a} di Napoli Federico II \\ Via Monteoliveto, 3 \\
I-80134 Napoli, Italy \\
and Istituto per le Applicazioni del Calcolo
``Mauro Picone", sezione di Napoli \\
Consiglio Nazionale delle Ricerche \\
via Pietro Castellino, 111 \\
I-80131 Napoli, Italy}
\email{fiorenza@unina.it}


\begin{abstract}
We present a unified strategy to derive Hardy--Poincar{\'e} inequalities on bounded and unbounded domains.  The approach allows proving a general Hardy--Poincar{\'e} inequality from which the classical Poincaré and Hardy inequalities immediately follow. The idea also applies to the more general context of variable exponent Sobolev spaces. The argument, concise and constructive, does not require a priori knowledge of compactness results and retrieves geometric information on the best constants.
\end{abstract}

\maketitle

\section{Introduction}

{\noindent}In its classical form, Poincar{\'e} inequality states that if $\Om$
is an open and {\tmem{bounded}} subset of $\RR^{\Ndim}$, $\Ndim \geqslant 1$,
and $\pp \in [1, \infty)$, then there exists a positive constant $c_{\Om,
\pp}$, depending only on $\pp$ and $\Om$ (in particular, on $\Ndim$), such that
\begin{equation}
\left( \int_{\Om} | u (x) |^{\pp} \mathd x \right)^{1 / \pp} \,
\leqslant  \, c_{\Om, \pp } \left( \int_{\Om} \left| \grad u (x) 
\right|^{\pp} \mathd x \right)^{1 / \pp} \quad \forall u \in C^{\infty}_c (\Om
), \label{eq:PoincareIntro}
\end{equation}
where $C^{\infty}_c (\Om )$ is the space of infinitely differentiable functions with
compact support in $\Om$, and $\left| \grad u \right|^{\pp} = \sum_{i =
1}^{\Ndim} | \partial_i u |^{\pp}$.

The multidimensional Hardy inequality
states that if $\Om$ is an open subset of $\RR^{\Ndim}$ (possibly
{\tmem{unbounded}}), $\pp \in [1, \infty)$, and $\Ndim > \pp$ (therefore,
necessarily $\Ndim \geqslant 2$), then there exists a positive constant
$c_{\Om, \pp }$, depending only on $\Om$ and $\pp$, such that
\begin{equation}
\left( \int_{\Om} \frac{| u (x) |^{\pp}}{| x |^{\pp}} \mathd x
\right)^{\frac{1}{\pp}}  \, \leqslant  \, c_{\Om, \pp
} \left( \int_{\Om} \left| \grad u (x) \right|^{\pp} \mathd x
\right)^{\frac{1}{\pp}} \quad \forall u \in C^{\infty}_c (\Om )
. \label{eq:HardyIntro}
\end{equation}
The version of {\eqref{eq:PoincareIntro}} reported here dates back to
Steklov~{\cite{Kuznetsov2014}} but originates in the work of
Poincar{\'e}~{\cite{Poincare1890}}, where the inequality is established in the
class of smooth functions that have zero mean in $\Om$. Likewise,
{\eqref{eq:HardyIntro}} dates back to Leray {\cite{Leray1933}} and is the
multidimensional analog of the one-dimensional Hardy
inequality~{\cite{Hardy1901}} (see
{\cite{Kuznetsov2014,Kufner2006,Naumann2010,Ruzhansky2019}} for some
historical details).

Over the years, {\eqref{eq:PoincareIntro}} and
{\eqref{eq:HardyIntro}} have been intensively investigated, and various
extended and refined versions have been derived to cover different
functional settings (see, e.g., \cite{Muckenhoupt1972, Brezis1997, Brezis2000, Dolbeault2012, Acosta2017}). Also, the knowledge
of sharp constants in such functional inequalities (e.g., the minimal
constants $c_{\Om, \pp }$ for which {\eqref{eq:PoincareIntro}} or
{\eqref{eq:HardyIntro}} holds), or even upper bounds on the optimal constants
which explicitly highlight their dependence on the geometry of the domain, have
remarkable applications in the Analysis of PDEs and Numerics (see, e.g.,
{\cite{Michlin1981,Dautray1990,Chou1993,GarciaAzorero1998,Verfuerth1999, Vazquez2000, Gazzola2004,Blanchet2007,Di_Fratta_2012,Kuznetsov2015,DiFratta2019}}). 
The literature on Hardy and Poincaré inequalities is endless; for further sources and developments, we refer the reader to the comprehensive presentations in
{\cite{Kufner2017,Ruzhansky2019,Edmunds2004}} for Hardy inequalities and
{\cite[Chap.~4]{Ziemer89}} for Poincar{\'e}-type inequalities.

This paper aims to present a unified strategy to derive Hardy--Poincar{\'e}
inequalities through a concise argument based on the divergence theorem. Our
approach extends to the more general context of variable exponent Sobolev
spaces {\cite{Diening2011,CruzUribe2014}}. Unified frameworks to treat such
inequalities are of some interest. In {\cite{Ziemer89}} a whole
chapter is devoted to a unified approach to Poincar{\'e} inequalities. The
underlying argument proceeds by contradiction and is based on
Rellich--Kondrachov compactness theorem; while the approach is both elegant
and simple, it relies on a compactness argument that requires workarounds to deal with unbounded domains where, in general, one lacks compactness. Moreover, the nonconstructive nature of the
approach produces a loss of information on the geometric content of the
Poincar{\'e} constant, which can be only partially recovered by scaling
arguments. Instead, our approach, more in the spirit of
{\cite{Avkhadiev2007,Barbatis2004,Gazzola2004,Takahashi2015}}, is constructive: it does not require a priori knowledge of compactness results and returns geometric
information on the Hardy--Poincaré constants. In some cases, it also provides the best possible constant (cf.~Remark~\ref{rmk:Hardysharp})

Our arguments' main idea came when the authors were analyzing some aspects of
the so-called {\tmem{demagnetizing field}}, which is the primary source of
nonlocal interactions in the variational theory of micromagnetism
{\cite{DiFratta2016,DiFratta2020}}. For a magnetization $m \in L^2(
\RR^3, \RR^3)$ the demagnetizing field $h [m]$ induced by $m$ is given
by $h [m] \assign \nabla u_m$ with $u_m$ the demagnetizing potential which
solves
\begin{equation}
- \Delta u_m \eqs \tmop{div} m \text{ \ in \ } \mathcal{D}' (\RR^3) . \label{eq:potential}
\end{equation}
Lax-Milgram lemma guarantees that equation (\ref{eq:potential}) possesses a
unique solution in the Beppo Levi space $B_{} L^1_0 ( \RR^3 )$,
where for every open subset $\Om \subseteq \RR^3$
({\tmabbr{cf.}}~{\cite[Chap.~XI, Part~B]{Dautray1990}}, see also
{\cite{Deny1954,Deny1955,MR1611383}})
\[ B_{} L^1 (\Om ) \assign \left\{ u \in \mathcal{D}' (\Om
) \; \suchthat \; \frac{u (\cdot)}{\sqrt{1 + \left| \, \cdot \,
\right|^2}} \in L^2 (\Om ) \text{ \ and \ } \nabla u \in L^2
(\Om, \RR^3 ) \right\}, \]
and $B_{} L^1_0 (\Om )$ is the closure of $C_c^{\infty} (\Om
)$ in $B_{} L^1 (\Om )$. These spaces are Hilbert spaces
when endowed with the norm:
\begin{equation}
\| u \|^2_{B_{} L^1 (\Om )} = \int_{\Om} \frac{| u (x) |^2}{1 +
| x |^2} \mathd x  + \int_{\Om} | \nabla u (x) |^2 \mathd x .
\label{eq:gradnorm}
\end{equation}
The only (a~priori) nontrivial point in the application of the Lax--Milgram
lemma to the weak formulation of {\eqref{eq:potential}} is in the observation
that the gradient seminorm $\left\| \grad u \right\|_{L^2 (\Om )}$
is actually a norm on $B_{} L^1_0 (\Om )$ equivalent to
{\eqref{eq:gradnorm}}. For that, one has to prove something less than Hardy
inequality {\eqref{eq:HardyIntro}}, namely that
\begin{equation}
\int_{\Om} \frac{| u (x) |^2}{1 + | x |^2} \mathd x \leqslant c_{\Om, \pp
}^2 \int_{\Om} | \nabla u (x) |^2 \mathd x \quad \forall u \in C_c^{\infty}
(\Om ) . \label{eq:BLineq}
\end{equation}
A proof of {\eqref{eq:BLineq}} is given in \ {\cite[Thm.~1,
p.~114]{Dautray1990}} where these weighted Sobolev spaces are covered in the
three-dimensional case and applied to the study of integral equations
associated with elliptic boundary value problems in exterior domains of
$\RR^3$. Again, the proof is by contradiction and based on a localization
argument that allows for compactness results.

The attempt to find a simplified proof of {\eqref{eq:BLineq}} led us to
Theorem~\ref{thm:mainfixexpomgen}. An accurate bibliographic search reveals
that, when $\pp = 2$, divergence theorem is adopted in {\cite[Thm~7.61, p.~466]{Salsa2016}} to
give a proof of the classical Poincar{\'e} inequality
{\eqref{eq:PoincareIntro}} in $H^1_0(\Om)$  and is also recognizable behind
some computations in {\cite{Payne1960}} (again in the case $\pp = 2$). Our
paper moves forward in this direction and shows how the idea can be used to
derive Hardy--Poincar{\'e} inequalities in more general settings: Poincar{\'e}
inequalities on domains bounded in one direction, weighted Hardy--Poincar{\'e}
inequalities on general domains, and Poincar{\'e} inequalities in variable
exponent Sobolev spaces.
In particular, in
Theorem~\ref{thm:VASBPI}, we derive modular Hardy--Poincar{\'e} inequalities
suited for variable exponent Sobolev spaces. Our results complement the remarkable findings obtained in {\cite[Thm.~3.1]{Fan2005}} which, roughly speaking, claim that Poincaré inequality cannot hold in a bounded domain $\Om$ of $\RR^N$ if the variable exponent $\pp (\cdot)$ is radial with respect to a point $x_0 \in \RR^N$ (i.e., $\pp (x) = \pp (| x - x_0 |)$) and the profile of $\pp$ is
decreasing. Instead, we show that this is the case provided that one restricts the class of competitors to the space of functions $u \in C_c^{\infty} (\Om)$ such
that $| u | \leqslant 1$ in $\Om$. Moreover, our result holds even if $\Om$ is unbounded, and returns explicit dependence of the modular Poincaré constant in terms of the geometry of the domain and $\pp(\cdot)$.

The paper is organized as follows. In Section~\ref{sec:PIBOneDirection}, we
give a short proof of the Poincar{\'e} inequality on domains that are bounded
in one direction. The result is well-known, but the proof we present allows us to show
the main ideas in a clear way, and also to discuss some issues on the
optimality of the Poincar{\'e} constant. In Section~\ref{sec:HPinequalities},
we present a general Hardy--Poincar{\'e} inequality from which well-known
functional inequalities follow as corollaries. In particular, the sharp
multidimensional inequality in $\RR^{\Ndim}$. Finally, in
Section~\ref{sec:HPinequVarSobsp}, we derive modular Hardy--Poincar{\'e} inequalities
suited for variable exponent Sobolev spaces.

\section{Poincar{\'e} inequality on domains bounded in one
direction}\label{sec:PIBOneDirection}
\begin{figure}[t]
\includegraphics[scale=1]{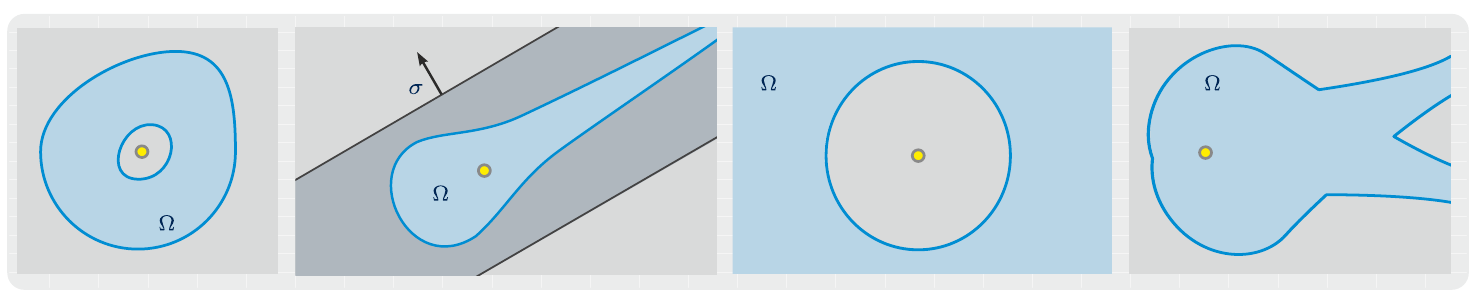}
\caption{Geometrically, a domain is bounded along the $\sigma$-direction if it is absorbed by a strip whose boundary (consisting of two parallel hyperplane) is perpendicular to $\sigma$. From left to right, examples of domains in $\RR^2$:
$\Om$ bounded (in all directions); $\Om$ unbounded but bounded in the $\sigma$ direction; $\Om$ unbounded (and not bounded in any direction); $\Om$ unbounded, but representable as the union of two open sets that are bounded in one direction.}
 \label{Fig:BD}
\end{figure}
{\noindent}Let $\sigma \in \Stwo^{\Ndim - 1}$. We say that an open subset $\Om
\subseteq \RR^{\Ndim}$ is bounded along the $\sigma$-direction if
\begin{equation}
\sup_{x \in \Om} | x \cdot \sigma | < + \infty .
\end{equation}
We say that $\Om$ is bounded in {\tmem{one}} direction if there exists a
direction $\sigma \in \Stwo^{\Ndim - 1}$ along which $\Om$ is bounded. Geometrically, a domain is bounded along the $\sigma$-direction if it is absorbed by a strip whose boundary (consisting of two parallel hyperplane) is perpendicular to $\sigma$ (see~Figure~\ref{Fig:BD}).

It is well-known that if $\Om$ is bounded along a direction, then the
Poincar{\'e} inequality holds in $W^{1, p}_0 (\Om )$ (see
{\cite[Theorem~5.3.1, p.~161]{Attouch2014}}. Precisely, if the open subset $\Om
\subseteq \RR^{\Ndim}$ is bounded in the $\sigma$-direction for some $\sigma
\in \Stwo^{\Ndim - 1}$, then there exists a constant $c_{\Om}(\pp) > 0$, depending only on $\pp$ and $\Om$ (in particular, on $\Ndim$), such that
\begin{equation}
\| u \|_{L^{\pp} (\Om )}  \leqslant c_{\pp, \Om} \left\|
\grad u \right\|_{L^{\pp} (\Om )} \quad \forall u \in W^{1, p}_0
(\Om ) \label{eq:PoincIneqButtazzo}
\end{equation}
with
\begin{equation}
\left\| \grad u \right\|_{L^{\pp} (\Om )}^{\pp} = \sum_{i \in
\NN_{\Ndim}} \int_{\Om} | \partial_i u |^{\pp} \mathd x.
\end{equation}
A common proof of {\eqref{eq:PoincIneqButtazzo}} consists of an argument by
contradiction and is based on Rellich--Kondrachov compactness theorem (see,
e.g., {\cite{Brezis2011,Ziemer89}}). Another way to infer
{\eqref{eq:PoincIneqButtazzo}} is by aligning the $\sigma$-direction to
one of the coordinates axes, and then applying the one-dimensional fundamental
theorem of calculus to get estimates of $u$ in terms of its
first-order derivatives (see, e.g., {\cite{Attouch2014,Leoni2017}}). From this
perspective, we can say that our approach is based on the fundamental theorem
of calculus in $\Ndim$-dimensions and applied to a vector-weighted version of
$u$, which allows simplifying the computations and giving unified arguments
regardless of $\Om$ being bounded or not. To clarify what we mean, we give a
concise proof of a more general version of {\eqref{eq:PoincIneqButtazzo}},
which will clarify our arguments common strategy.

\begin{theorem}
\label{thm:Poincarebound1d}Suppose that $\Om \subseteq \RR^{\Ndim}$ is open
and bounded in the $\sigma$-direction for some $\sigma \in \Stwo^{\Ndim -
1}$. Then, for every $\pp \in [1, + \infty)$ we have
\begin{equation}
\| u \|_{L^{\pp} (\Om )}  \leqslant \pp c_{\Om, \sigma} \|
\partial_{\sigma} u \|_{L^{\pp} (\Om )} \quad \forall u \in
W^{1, p}_0 (\Om ), \label{eq:PoincareDifFio}
\end{equation}
with
\begin{equation}
c_{\Om, \sigma} \assign \inf_{x_0 \in \RR^{\Ndim}} \sup_{x \in \Om} | (x -
x_0) \cdot \sigma | \label{eq:exprconstbounded1d}
\end{equation}
depending only on the projection of $\Om$ onto the $\sigma$-direction.
\end{theorem}

\begin{remark}
The domain on the right in Figure~\ref{Fig:BD} is not bounded in one direction. However, it is representable as a finite union of open sets that are bounded in one direction. For these domains, something can still be said, and we refer the reader to Agmon's book for details (cf.~\cite[Lemma~7.4, p.~75]{Agmon2010}).
\end{remark}

\begin{remark}
\label{rmk:fromsigmatograd} In particular, if $\Om$ is
bounded, then given any (not necessarily orthogonal) unit basis
$(\sigma_i)_{i \in \NN_{\Ndim}}$ of $\RR^{\Ndim}$, we get that
\begin{equation} \label{eq:partcasePI}
\| u \|_{L^{\pp} (\Om )} \, \leqslant
\, \pp c_{\Om, \sigma_{\ast}} \| \partial_{\sigma_{\ast}} u
\|_{L^{\pp} (\Om )} \, \leqslant \,
\pp c_{\Om, \sigma_{\ast}} \left\| \grad u \right\|_{L^{\pp} (\Om
)},
\end{equation}
with $\sigma_{\ast} \assign \tmop{argmin}_{i \in \NN_{\Ndim}} c_{\Om,
\sigma_i}$, as well as
\begin{equation}
\| u \|_{L^{\pp} (\Om )} \, \leqslant
\, \frac{\pp}{\Ndim^{1 / \pp}} \left( \sum_{i \in
\NN_{\Ndim{\color[HTML]{005500}}}} c_{\Om, \sigma_i}^{\pp} \|
\partial_{\sigma_i} u \|_{L^{\pp} (\Om )}^{\pp} \right)^{1 /
\pp} \, \leqslant \, \frac{\pp}{\Ndim^{1 / \pp}}
\left( \max_{i \in \NN_{\Ndim}} c_{\Om, \sigma_i} \right) \left\| \grad u
\right\|_{L^{\pp} (\Om )},
\end{equation}
with $\grad u = \left( \partial_{\sigma_1} u, \ldots,
\partial_{\sigma_{\Ndim}} u \right)$. Note that if $\Om$ is {\tmem{connected}} then for every $i \in \NN_N$ the projection of $\Om$ onto $\sigma_i$ defined by
\begin{equation}
\Pi_{\sigma_i} (\Om ) = \left\{ t \in \RR \of t = x \cdot
\sigma_i \text{ for some } x \in \Om \right\}
\end{equation} 
is an interval and, since
$\Pi_{\sigma_i}$ is a nonexpansive map, we have that
\begin{equation}
\tmop{diam} \Pi_{\sigma_i} (\Om ) \leqslant \tmop{diam} \Om .
\label{eq:diamOm}
\end{equation}
Therefore, if we set $c_{\Om,\sigma_\star}:=\max_{i\in\NN_n} c_{\Om,\sigma_i}$, denote by
$\eta_{\star} \in \Pi_{\sigma_{\star}} (\Om )$ the center of the
interval $\Pi_{\sigma_{\star}} (\Om )$, and by $y_{\star}$ one of the elements of $\RR^\Ndim$ such that $\Pi_{\sigma_{\star}} (y_{\star}) = \eta_{\star}$, i.e., such that $y_\star \cdot \sigma_\star=\eta_\star$, by
{\eqref{eq:diamOm}} we infer that
\begin{equation}
c_{\Om,\sigma_\star}   \leqslant   \sup_{x \in \Om} | (x - y_{\star})
\cdot \sigma_\star |\eqs \sup_{x \in \Om} | x \cdot \sigma_{\star} - \eta_{\star} | \eqs  \frac{\tmop{diam} \Pi_{\sigma_{\star}} (\Om )}{2}
 \leqslant  \frac{\tmop{diam} \Om}{2} .
\end{equation}
Overall, by \eqref{eq:partcasePI} and \eqref{eq:diamOm}  we get that if $\Om$ is {\tmem{connected}} then, for every $u \in
C^{\infty}_c (\Om )$ there holds
\begin{equation}
\| u \|_{L^{\pp} (\Om )} \, \leqslant
\, \frac{\pp}{\Ndim^{1 / \pp}} \frac{\tmop{diam} \Om}{2}
\left\| \grad u \right\|_{L^{\pp} (\Om )} .
\label{eq:Poincaresharpconn}
\end{equation}
The previous relation is the analog of the well-known sharp estimate
obtained for $\pp = 1$ in {\cite{Acosta2004}} in the class of
{\tmem{convex}} domains having a prescribed diameter. Note, however, that in
{\cite{Acosta2004}} the optimal estimate concerns the class of functions
with zero mean in $\Om$; here, we consider the class of functions vanishing
on $\partial \Om$. Nevertheless, the main difficulty in {\cite{Acosta2004}}
is the sharpening of the Poincar{\'e} constant from $\tmop{diam} \Om$ to
$\frac{1}{2} \left( \tmop{diam} \Om \right)$; this is due to the well-known
obstruction that, in general, in dimension $\Ndim \geqslant 2$, it is not true
that a convex subset of $\RR^{\Ndim}$ is contained in a ball of radius half
its diameter. Although our result overcomes this issue, there is
no pretense of optimality in {\eqref{eq:Poincaresharpconn}}.
\end{remark}

\begin{proof}
By density, it is sufficient to prove {\eqref{eq:PoincareDifFio}} for
every $u \in C_c^{\infty} (\Om)$. First, assume that $\pp > 1$
so that $| u |^{\pp} \in C^1_c (\Om)$ and (with $\tmop{sign} t
= 1$ if $t \geqslant 0$ and $\tmop{sign} t = -1$ otherwise) we have
\begin{equation}
\grad | u |^{\pp} = \pp (\tmop{sign} u) | u |^{\pp - 1} \grad u.
\end{equation}
For any unit vector $\sigma \in \Stwo^{\Ndim - 1}$, and any $x_0 \in
\RR^{\Ndim}$, we set $\tau_{x_0} (x) \assign (x - x_0)$. Observe that
$\divv_x ((\tau_{x_0} \cdot \sigma) \sigma) = \divv_x ((x \cdot \sigma)
\sigma) = | \sigma |^2 = 1$, therefore, pointwise in $\Om$ we have that
\begin{equation}
\divv \left[ (\tau_{x_0} \cdot \sigma) \sigma | u |^{\pp} \right]  \eqs 
| u |^{\pp} + \pp (\tmop{sign} u) (\tau_{x_0} \cdot \sigma) | u |^{\pp -
1} \cdot \partial_{\sigma} u.
\end{equation}
After that, integrating both sides of the previous relation on $\Om$, using
the divergence theorem and then H{\"o}lder inequality, we get that
\begin{align}
\int_{\Om} | u |^{\pp} \mathd x & =  - \pp \int_{\Om} (\tmop{sign} u) | u
|^{\pp - 1} (\tau_{x_0} \cdot \sigma) \partial_{\sigma} u \mathd x \\
& \leqslant  \pp \sup_{x \in \Om} | \tau_{x_0} \cdot \sigma | \left(
\int_{\Om} | u |^{\pp} \mathd x \right)^{1 / \qq} \left( \int_{\Om} |
\partial_{\sigma} u |^{\pp} \mathd x \right)^{1 / \pp},
\end{align}
where we denoted by $\qq$ the conjugate exponent of $\pp$, i.e., $1 / \pp +
1 / \qq = 1$. By the arbitrariness of $x_0 \in \RR^{\Ndim}$, we get
{\eqref{eq:PoincareDifFio}} for every $\pp > 1$. Eventually, since
$c_{\Om, \sigma}$ does not depend on $\pp$, if $u \in C_c^{\infty}(
\Om)$, passing to the limit for $\pp \rightarrow 1^+$ in
{\eqref{eq:PoincareDifFio}} we conclude.
\end{proof}

\section{Hardy--Poincar{\'e} inequalities in arbitrary open
sets}\label{sec:HPinequalities}

{\noindent}In this section, we prove a weighted inequality that applies to
general open subsets of $\RR^{\Ndim}$. The argument is short and constructive. Moreover,
it is sharp in the sense that, as a particular case, we get the classical
Hardy inequality in $C_c^{\infty} \left( \RR^{\Ndim} \right)$ with optimal
constant (see Remark~\ref{rmk:Hardysharp}). 
To the best of the authors' knowledge, the inequality we present is new, but in any case, the point we want to emphasize here resides in the short argument used to derive it. 

\begin{theorem}
\label{thm:mainfixexpomgen}Let $\Om$ be an open subset
{\tmem{(}}bounded or not{\tmem{)}} of $\RR^{\Ndim}$, $\Ndim\geqslant 1$, and $\pp \in [1, + \infty)$. For any
$\lambda > 0, \alpha > 0, \beta \neq 0$ in the set of real numbers and such
that $\alpha \beta < \Ndim$, and for arbitrary $x_0 \in \RR^{\Ndim}$, we
have that
\begin{equation}
\int_{\Om} \frac{| u (x) |^{\pp}}{(\lambda + | x - x_0
|^{\alpha})^{\beta}} \mathd x  \leqslant 
\left(\frac{\pp}{\kappa_+}\right)^\pp \int_{\Om} \left| \grad u (x) \right|^{\pp}
\frac{| x - x_0 |^{\pp}}{(\lambda + | x - x_0 |^{\alpha})^{\beta}} \mathd
x \quad \forall u \in C^{\infty}_c (\Om) . \label{eq:gen1}
\end{equation}
Here, the strictly positive constant $\kappa_+$ is given by $\kappa_+
\assign \left( \Ndim - \alpha \beta \right)$ if $\beta > 0$ and $\kappa_+ =
\Ndim$ if $\beta < 0$.
\end{theorem}

\begin{remark}
For $\pp = \alpha = 2$, $\lambda = \beta = 1$, $x_0 = 0$, and in the
three-dimensional case $\Ndim = 3$, {\eqref{eq:gen1}} reduces to the inequality
\begin{equation}
\int_{\Om} \frac{| u (x) |^{2}}{1 + | x |^2} \mathd x  \leqslant 
4  \int_{\Om} \left| \grad u (x) \right|^{2}
\frac{| x |^{2}}{1 + | x |^{2}} \mathd
x  \quad \forall u \in C^{\infty}_c (\Om) \, ,
\end{equation}
which is a sharper version of the Hardy--Poincar{\'e} inequality {\eqref{eq:BLineq}} in the Beppo Levi space
$B_{} L^1_0 (\Om)$ which plays an important role in the
analysis of elliptic problems in exterior domains. As already recalled in
the introduction,  the proof of
{\eqref{eq:BLineq}} given in {\cite[Thm.~1, p.~114]{Dautray1990}} is by contradiction and based on a localization and
compactness argument.
\end{remark}

\begin{remark}
Note that $\kappa_+$ is a {\tmem{universal constant}} in the sense that it
depends only on the dimension $\Ndim$ of the ambient space if $\beta < 0$,
as well as on the decay-at-infinity rate $\alpha \beta$ when $\beta > 0$. In
particular, $\kappa_+$ does not depend on $\Om$. Also, we point out that a
closer look at the proof of Theorem~\ref{thm:mainfixexpomgen}
({\tmabbr{cf.}}~{\eqref{eq:tobeused}}) shows that when $\beta < 0$, one can
be more accurate and set
\[ \kappa_+ \assign 
\begin{cases}
\Ndim - \alpha | \beta | & \text{if } \beta > 0,\\
\Ndim + \alpha | \beta | \min_{x \in \Om} \frac{| x - x_0
|^{\alpha}}{\lambda + | x - x_0 |^{\alpha}} & \text{if } \beta < 0 .
\end{cases}  \]
However, this sharper definition of $\kappa_+$, makes the constant no more
universal when $\beta < 0$ and $\Om$ is such that $\min_{x \in \Om} | x -
x_0 | > 0$ for some $x_0 \in \RR^{\Ndim}$ (in particular, if $\Om$ is
bounded).
\end{remark}

\begin{proof}[Proof of Theorem~\ref{thm:mainfixexpomgen}]
In what follows, $x_0 \in \RR^{\Ndim}$ denotes an arbitrary point. With
$\lambda > 0, \alpha > 0, \beta \neq 0$ in the set of real numbers, and $\alpha \beta<\Ndim$, we
define the weight
\begin{equation}
\omega (x) \assign \frac{1}{(\lambda + | \tau_{x_0} (x)
|^{\alpha})^{\beta}}, \quad \tau_{x_0} (x) \assign x - x_0 .
\end{equation}
We note that, pointwise in $\Om$,
\begin{equation}
\divv_x [\omega \tau_{x_0}] (x) \eqs \left( \Ndim - \alpha \beta \frac{|
\tau_{x_0} (x) |^{\alpha}}{\lambda + | \tau_{x_0} (x) |^{\alpha}} \right)
\omega (x) \; \geqslant \; \kappa_+ \omega (x), \label{eq:tobeused}
\end{equation}
with $\kappa_+ = \left( \Ndim - \alpha \beta \right)$ if $\beta > 0$ and
$\kappa_+ = \Ndim$ if $\beta < 0$. Note that, by assumption, $\Ndim > \alpha
\beta$ and, therefore, $\kappa_+ > 0$ regardless of the value of the
parameters $\alpha > 0$ and $\beta \neq 0$. Also, note that for easier
readability, we do not report in the notation the dependences of $\omega$
and $\kappa_+$ on $\lambda, \alpha, \beta$, $\Ndim$, and $x_0$.

First, assume that $\pp > 1$ so that $| u |^{\pp} \in C^1_c (\Om)$ and (with $\tmop{sign} t = 1$ if $t \geqslant 0$ and $\tmop{sign} t
= -1$ otherwise) we have $\grad | u |^{\pp} = \pp (\tmop{sign} u) | u |^{\pp
- 1} \grad u$. Then, we lower bound the divergence of $\omega | u |^{\pp}
\tau_{x_0}$ as follows:
\begin{align*}
\divv_x \left[ \omega | u |^{\pp} \tau_{x_0} \right] & \eqs  | u |^{\pp}
\divv_x [\omega \tau_{x_0}] + \pp (\tmop{sign} u) | u |^{\pp - 1} \omega
\grad u \cdot \tau_{x_0} \nonumber\\
& \overset{\eqref{eq:tobeused}}{\geqslant}  \kappa_+
\omega | u |^{\pp} + \pp (\tmop{sign} u) | u |^{\pp / \qq} \omega \grad u
\cdot \tau_{x_0}, \nonumber
\end{align*}
where we denoted by $\qq$ the conjugate exponent of $\pp$, i.e., $1 / \pp +
1 / \qq = 1$. By the divergence theorem, integrating over $\Om$ the first
and last member of the previous expression, we obtain
\begin{equation}
\kappa_+ \int_{\Om} | u |^{\pp} \omega \mathd x \leqslant \pp \int_{\Om}
\left( | u |^{\pp / \qq} \omega^{1 / \qq} \right) \left(\omega^{1 / \pp}
\left| \grad u \right| \cdot | \tau_{x_0} | \right) \mathd x .
\label{eq:fundsteppoincare}
\end{equation}
By H{\"o}lder inequality, we get
\begin{equation}
\kappa_+ \int_{\Om} | u |^{\pp} \omega \mathd x  \leqslant  \pp \left(
\int_{\Om} | u |^{\pp} \omega^{} \mathd x \right)^{\frac{1}{\qq}} \left(
\int_{\Om} \left| \grad u \right|^{\pp} | \tau_{x_0} |^{\pp} \omega^{}
\mathd x \right)^{\frac{1}{\pp}},
\end{equation}
that is, in more explicit terms,
\begin{equation}
\int_{\Om} \frac{| u (x) |^{\pp}}{(\lambda + | x - x_0
|^{\alpha})^{\beta}} \mathd x^{}  \leqslant  \left( \frac{\pp}{\kappa_+}
\right)^{\pp} \int_{\Om} \left| \grad u (x) \right|^{\pp} \frac{| x - x_0
|^{\pp}}{(\lambda + | x - x_0 |^{\alpha})^{\beta}} \mathd x.
\end{equation}
This concludes the proof under the assumption that $\pp > 1$. The case $\pp
= 1$ can be obtained by noting that since $u \in C^{\infty}_c (\Om)$, Lebesgue dominated convergence theorem permits to pass to the
limit for $\pp \rightarrow 1^+$ in the previous expression.
\end{proof}

\subsection{Corollaries}The general Poincar{\'e}-type inequality in
Theorem~\ref{thm:mainfixexpomgen} unifies many well-known inequalities on
bounded and unbounded domains. The first consequence is a generalized version
of the classical multidimensional Hardy
inequality
{\cite{Hardy1952,Kufner2017,Ruzhansky2019}}, which follows as a particular
case of Theorem~\ref{thm:mainfixexpomgen} when $\beta > 0$.

\begin{corollary}
Let $\Om$ be an open subset
{\tmem{(}}bounded or not{\tmem{)}} of $\RR^{\Ndim}$, $\Ndim\geqslant 1$, and $\pp \in [1, + \infty)$. For any $0 < \gamma < \Ndim$ \ and
for arbitrary $x_0 \in \RR^{\Ndim}$, the following Hardy-type inequality
holds:
\begin{equation}
 \int_{\Om} \frac{| u (x) |^{\pp}}{| x - x_0 |^{\gamma}} \mathd x
  \leqslant \left(\frac{\pp}{\Ndim - \gamma}\right)^\pp
\int_{\Om} \frac{\left| \grad u (x) \right|^{\pp}}{| x - x_0 |^{\gamma -
\pp}} \mathd x  \quad \forall u \in C^{\infty}_c
(\Om) . \label{eq:gen2}
\end{equation}
In particular, for $\gamma \assign \pp < \Ndim$ we get a generalized version
of Hardy inequality in $\RR^{\Ndim}$:
\begin{equation}
 \sup_{x_0 \in \RR^N} \int_{\Om} \frac{| u (x) |^{\pp}}{| x - x_0
|^{\pp}} \mathd x  \leqslant  \left(\frac{\pp}{\Ndim -
\pp}\right)^\pp \int_{\Om} \left| \grad u (x) \right|^{\pp} \mathd x
 \quad \forall u \in C^{\infty}_c(\Om), \label{eq:gen3}
\end{equation}
which reduces to the classical {\tmem{(}}sharp{\tmem{)}} Hardy inequality
when $\Om \assign \RR^{\Ndim}$:
\begin{equation}
 \int_{\RR^{\Ndim}} \frac{| u (x) |^{\pp}}{| x |^{\pp}} \mathd x
  \leqslant  \left(\frac{\pp}{\Ndim - \pp}\right)^\pp 
\int_{\RR^{\Ndim}} \left| \grad u (x) \right|^{\pp} \mathd x
 \quad \forall u \in C^{\infty}_c (\Om) . \label{eq:gen3.1}
\end{equation}
\end{corollary}

\begin{remark}
\label{rmk:Hardysharp}We stress that the constant $\pp / \left( \Ndim - \pp
\right)$ in {\eqref{eq:gen3.1}} is known to be sharp (see, e.g.,
{\cite[Chap.~2]{Ruzhansky2019}}). Also, note that while {\eqref{eq:gen2}}
makes sense in any spatial dimension $\Ndim$, the condition $\Ndim > \pp
\geqslant 1$ in {\eqref{eq:gen3}} and {\eqref{eq:gen3.1}} forces space
dimensions $\Ndim \geqslant 2$.
\end{remark}

\begin{proof}
Consider {\eqref{eq:gen1}} with $\beta > 0$. Taking the $\liminf$ for
$\lambda \rightarrow 0^+$ of both sides of {\eqref{eq:gen1}} and invoking
Beppo Levi's monotone convergence theorem, we get {\eqref{eq:gen2}} with
$\gamma \assign \alpha \beta$. Finally, recall that $\kappa_+ = \Ndim -
\alpha | \beta |$ when $\beta > 0$.
\end{proof}

The second consequence we want to establish is a dual version of Hardy
inequality.
It follows as a particular case of Theorem~\ref{thm:mainfixexpomgen} when
$\beta < 0$.

\begin{corollary}
Let $\Om \subseteq \RR^{\Ndim}$ be an open set {\tmem{(}}bounded or
not{\tmem{)}}, $\pp \in [1, + \infty)$. For any $\gamma > 0$ \ and for
arbitrary $x_0 \in \RR^{\Ndim}$, the following Hardy-type inequality holds:
\begin{equation}
 \int_{\Om} | u (x) |^{\pp} | x - x_0 |^{\gamma} \mathd x
  \leqslant  \left(\frac{\pp}{\Ndim}\right)^\pp  \int_{\Om}
\left| \grad u (x) \right|^{\pp} | x - x_0 |^{\pp + \gamma} \mathd x
 \quad \forall u \in C^{\infty}_c (\Om) . \label{eq:gen2betaneg}
\end{equation}
In particular,
\begin{equation}
 \int_{\Om} | u (x) |^{\pp} \mathd x 
\leqslant  \left(\frac{\pp}{\Ndim}\right)^\pp \int_{\Om} \left| \grad u (x)
\right|^{\pp} | x - x_0 |^{\pp} \mathd x  \quad
\forall u \in C^{\infty}_c (\Om) . \label{eq:gen3betaneg}
\end{equation}
\end{corollary}

\begin{proof}
Consider {\eqref{eq:gen1}} with $\beta < 0$. We then have, for every $u \in
C^{\infty}_c (\Om)$,
\[ \int_{\Om} | u (x) |^{\pp} (\lambda + | x - x_0 |^{\alpha})^{| \beta |}
\mathd x \, \leqslant \, \left( \frac{\pp}{\Ndim} \right)^{\pp}
\int_{\Om} \left| \grad u (x) \right|^{\pp} | x - x_0 |^{\pp} (\lambda +
| x - x_0 |^{\alpha})^{| \beta |} \mathd x, \]
because $\kappa_+ \assign \Ndim$ when $\beta < 0$. With $\gamma \assign
\alpha | \beta |$, taking the limit for $\lambda \rightarrow 0^+$ of both
sides of the previous inequality we get
\begin{equation}
\int_{\Om} | u (x) |^{\pp} | x - x_0 |^{\gamma} \mathd x  \leqslant 
\left( \frac{\pp}{\Ndim} \right)^{\pp} \int_{\Om} \left| \grad u (x)
\right|^{\pp} | x - x_0 |^{\pp + \gamma} \mathd x.
\end{equation}
Since $u \in C^{\infty}_c (\Om)$, by Lebesgue's dominated
convergence theorem we can pass to the limit for $\gamma \rightarrow 0^+$ to
obtain {\eqref{eq:gen3betaneg}}.
\end{proof}

For completeness, we observe that another consequence of
Theorem~\ref{thm:mainfixexpomgen} is the classical Poincar{\'e} inequality in
bounded domains of \ $\RR^{\Ndim}$ which, as already pointed out in
Remark~\ref{rmk:fromsigmatograd}, is also a consequence of
Theorem~\ref{thm:Poincarebound1d}.

\begin{corollary}
Let $\Om \subseteq \RR^{\Ndim}$ be a {\tmem{bounded}} open set, $\pp \in [1,
+ \infty)$. The following Poincar{\'e}-type inequality
\begin{equation}
\left( \int_{\Om} | u (x) |^{\pp} \mathd x \right)^{\frac{1}{\pp}} 
\leqslant  \frac{\pp}{\Ndim} c_{\Om} \left( \int_{\Om} \left| \grad u
(x) \right|^{\pp} \mathd x \right)^{\frac{1}{\pp}} \quad \forall u \in
W_0^{1, \pp} (\Om) \label{eq:gen4}
\end{equation}
with
\begin{equation}
c_{\Om} \assign \inf_{x_0 \in \RR^N} \sup_{x \in \Om} | x - x_0 |
\leqslant \tmop{diam} \Om .
\end{equation}
\end{corollary}

\begin{proof}
By density, it is sufficient to assume that $u \in C^{\infty}_c (\Om)$. Since now $\Om$ is assumed bounded, from {\eqref{eq:gen2}}, we get
that when $0 < \gamma < \Ndim$, for every $x_0 \in \RR^{\Ndim}$ and every $u \in C_c^{\infty}(\Om)$ there holds
\begin{equation}
\frac{1}{\sup_{x \in \Om} | x - x_0 |^{\gamma}} \int_{\Om} | u (x) |^{\pp}
\mathd x \leqslant  \left( \frac{\pp}{\Ndim - \gamma} \right)^{\pp}
\sup_{x \in \Om} | x - x_0 |^{\pp - \gamma} \int_{\Om} \left| \grad u (x)
\right|^{\pp} \mathd x.
\end{equation}
Therefore,for $0 < \gamma < \min\{\pp,\Ndim\}$, we get
\begin{equation}
\int_{\Om} | u (x) |^{\pp} \mathd x \leqslant \left(
\frac{\pp}{\Ndim - \gamma} \right)^{\pp} \left( \sup_{x \in \Om} | x - x_0
| \right)^{\pp} \int_{\Om} \left| \grad u (x) \right|^{\pp} \mathd x.
\end{equation}
Eventually, taking the limit for $\gamma \rightarrow 0^+$ we conclude.
\end{proof}

\section{Hardy--Poincar{\'e} inequality in variable exponent Sobolev
spaces}\label{sec:HPinequVarSobsp}
{\noindent} In the variable exponents' framework, the classical Poincar{\'e} inequality \eqref{eq:PoincareIntro} can be seen either as a special case of the norm inequality
\begin{equation}
\| u \|_{L^{\pp (\cdot)} (\Om )}  \leqslant c \left\|
\grad u \right\|_{L^{\pp (\cdot)}} \quad \forall u \in C^{\infty}_c (\Om
), \label{eq:Poincvarnorm}
\end{equation}
or as a particular case of the modular inequality
\begin{equation}
 \int_{\Om} | u (x) |^{\pp (\cdot)} \mathd x  \,
\leqslant  \, c  \int_{\Om} \left| \grad u (x) 
\right|^{\pp (\cdot)} \mathd x \quad \forall u \in C^{\infty}_c (\Om
). \label{eq:Poincvarmodu}
\end{equation}
A straightforward consequence of the definition of the norm of variable exponent Lebesgue spaces (see, e.g., \cite[Def.~2.16, p.~20]{CruzUribe2014}) is that modular inequalities imply the corresponding norm inequalities; therefore, in general, the validity of a modular inequality requires the same or stronger assumptions on the exponent function.

The norm form \eqref{eq:Poincvarnorm} of the Poincar{\'e} inequality is known to hold in bounded domains provided that uniform continuity or small local oscillation on the exponent is assumed (see, e.g., \cite[Theorem~3.10]{Kova1991}, \cite[Theorem~2.11, p.~69]{Edmunds2014}, \cite[Theorem~8.2.18, p.~263]{Diening2011}, \cite[Theorem~6.2.8, p.~130]{Hasto2019}, \cite[Proposition~2.4]{Mercaldo2013}, \cite[Section~4]{Ciarlet2011}, \cite[Theorem~1.1]{Youssfi2019}). The norm form \eqref{eq:Poincvarnorm} also holds under certain regularity assumptions expressed in terms of the boundedness of the maximal operator (\cite[Theorem~6.21, p.~249]{CruzUribe2014}, \cite[Theorem~8.2.4, p.~255]{Diening2011}). Finally, we mention \cite[Example~3.5]{FioRemovability2020}, where the classical way to get the Poincar{\'e} inequality as a consequence of the Sobolev inequality is stated in terms of modulars, and \cite[Lemma~13.7]{Zhikov2011} where the same inequality is obtained under a compactness assumption.

The modular form \eqref{eq:Poincvarmodu} of the Poincar{\'e} is a broader issue. The interest in it, supported by a general grown attention to modular inequalities (see, e.g., \cite{Fiomodulars2021,CruzUribe2018}), is mainly because of various examples of its invalidity, which prompted for sufficient conditions for its validity (see, e.g., the counterexamples in \cite[Example,~p.~444]{Fan2001}, \cite[Example~8.2.7, p.~257]{Diening2011}, \cite[Section~2]{Maeda2008}; see also \cite{FioCategories2019} for an overview about typical problems in the variable exponents framework). Positive results are, e.g., \cite[Theorem~4.1]{Maeda2008}, \cite[Theorem~3.3]{Fan2005}, \cite[Theorem 1]{Allegretto2007}, (see also \cite[Proposition~8.2.8(a), p.~257]{Diening2011}, \cite[Theorem~6.2.10, p.~131]{Hasto2019}). Note that all the cited results assume the boundedness of the domain.

In this Section, we derive modular Hardy--Poincar{\'e} inequalities
suited for variable exponent Sobolev spaces, which complement the remarkable results obtained in {\cite{Fan2005}}. Indeed, as a consequence of \cite[Thm.~3.1]{Fan2005}, one obtains that there is no hope for a Poincaré inequality in a bounded domain $\Om$ of $\RR^N$ if the variable exponent $\pp (\cdot)$ is radial with respect to a point $x_0 \in \RR^N$,
i.e., $\pp (x) = \pp (| x - x_0 |)$, and the profile of $\pp$ is
decreasing. Instead, quite surprisingly, statement \tmem{ii}\@.~of Theorem~\ref{thm:VASBPI} below shows that this is the case provided that one restricts the class of competitors to the space of functions $u \in C_c^{\infty} (\Om)$ such
that $| u | \leqslant 1$ in $\Om$ (see next Remark~\ref{rmk:Fan}). Moreover, our result holds even if $\Om$ is unbounded, and returns explicit dependence of the modular Poincaré constant in terms of the geometry of the domain and $\pp(\cdot)$.

\begin{theorem} \label{thm:VASBPI}
Let $\Om \subseteq \RR^{\Ndim}$ be an open set {\tmem{(}}bounded or
not{\tmem{)}}, and suppose that the variable exponent
\[ \pp (\cdot) : \Om \rightarrow [1, + \infty) \]
is in $L^\infty (\Om)$. The following assertions hold:
\begin{enumerate}[label=\roman*.]
\item If $\pp (\cdot)$ is constant in the $\sigma$ direction, for some
$\sigma \in \Stwo^{N - 1}$, then for every $u \in C_c^{\infty} (\Om)$ there holds
\begin{equation}
\int_{\Om} | u (x) |^{\pp (x)} \mathd x \leqslant \kappa (\pp) \inf_{x_0 \in \RR^N} \int_{\Om} | \partial_{\sigma} u (x) |^{\pp
(x)} | (x - x_0) \cdot \sigma |^{\pp (x)} \mathd x. \label{eq:thvarLS1}
\end{equation}
for some positive constant which depends on $\pp (\cdot)$, given by
\begin{equation}
\kappa (\pp) \assign \sup_{x \in \Om} \left[ 2^{\pp (x)}
(\pp (x) - 1 )^{\pp (x) - 1} \right] .
\label{eq:exprconstplog}
\end{equation}
In particular, if $\Om$ is also bounded in the $\sigma$-direction, then
for every $u \in C_c^{\infty} (\Om)$ there holds
\begin{equation}
\int_{\Om} | u (x) |^{\pp (x)} \mathd x \leqslant \kappa_{\Om, \sigma}
(\pp) \int_{\Om} | \partial_{\sigma} u (x) |^{\pp (x)}
\mathd x
\end{equation}
for some positive constant $\kappa_{\Om, \sigma} (\pp )$ which
depends on the projection of $\Om$ on $\sigma$, and $\pp (\cdot)$:
\begin{equation}
\kappa_{\Om, \sigma} (\pp) \assign \kappa (\pp) \inf_{x_0 \in \RR^N} \sup_{x \in \Om} | (x - x_0) \cdot \sigma
|^{\pp (x)} . \label{eq:exprconstplog2}
\end{equation}
\item If $\pp (\cdot)$ is radial with respect to a point $x_0 \in \RR^N$,
i.e., $\pp (x) = \pp (| x - x_0 |)$, and the profile of $\pp$ is
decreasing, then, for every $u \in C_c^{\infty} (\Om)$ such
that $| u | \leqslant 1$ in $\Om$ there holds
\begin{equation} \label{eq:thmPoinclog0}
\int_{\Om} | u (x) |^{\pp (x)} \mathd x \leqslant \kappa (\pp) \inf_{(x_0, \sigma) \in \RR^N \times \Stwo^{N - 1}} \int_{\Om} |
\partial_{\sigma} u (x) |^{\pp (x)} | (x - x_0) \cdot \sigma |^{\pp (x)}
\mathd x
\end{equation}
with $\kappa (\pp)$ given by
{\tmem{{\eqref{eq:exprconstplog}}}}. In particular, if $\Om$ is bounded in
the $\sigma$ direction, then for every $u \in C_c^{\infty} (\Om)$ such that $| u | \leqslant 1$ in $\Om$ there holds
\begin{equation}
\int_{\Om} | u |^{\pp (x)} \mathd x \leqslant \kappa_{\Om, \sigma}
( \pp ) \int_{\Om} | \partial_{\sigma} u |^{\pp (x)} \mathd x
\label{eq:thmPoinclog}
\end{equation}
with $\kappa_{\Om, \sigma} (\pp)$ given by
{\tmem{{\eqref{eq:exprconstplog2}}}}.
\end{enumerate}
\end{theorem}

\begin{remark}
Clearly, in the hypotheses of {\tmem{ii.}}, if $\Om$ is bounded (in any
direction) then given any (not necessarily orthogonal) unit basis
$(\sigma_i)_{i \in \NN_{\Ndim}}$ of $\RR^{\Ndim}$, we get, by
{\eqref{eq:thmPoinclog}}, the following Poincar{\'e} inequality
\begin{equation}
\int_{\Om} | u (x) |^{\pp (x)} \mathd x \leqslant \kappa_{\Om} (\pp) \int_{\Om} \left| \grad u (x) \right|^{\pp (x)} \mathd x
\label{eq:thmPoincloggrad}
\end{equation}
where now $\left| \grad u \right|^{\pp (x)} \assign \sum_{i = 1}^{\Ndim} |
\partial_{\sigma_i} u |^{\pp (x)}$ and $\kappa_{\Om} (\pp) > 0$
is a positive constant linked to the constants $\kappa_{\Om, \sigma_i}$,
e.g., in the way already discussed in Remark~\ref{rmk:fromsigmatograd}.
\end{remark}

\begin{proof} We first show the results under the additional regularity assumption that $\pp(\cdot) \in C^1_b (\Om)$, where $C^1_b (\Om)$ is the space of
continuously differentiable functions on $\Om$ whose partial derivatives are,
again, continuous and bounded in $\Om$. Later we use a density argument to extend the results to the general case of $\pp(\cdot)\in L^\infty (\Om)$.

Let $u \in C^{\infty}_c (\Om)$ and $\pp \in C^1_b (\Om)$. First, assume that $\pp (x) > 1$ for every $x \in \Om$. This
assures that $| u |^{\pp (\cdot)}$ is in $C^1_c (\Om)$. For any
$\sigma \in \Stwo^{N - 1}$ we have $\divv_x ((x \cdot \sigma) \sigma) = |
\sigma |^2 = 1$ and, therefore,
\begin{align}
\divv_{x \,} \left[ ((x - x_0) \cdot \sigma) \sigma | u (x) |^{\pp (x)}
\right] & \eqs  | u (x) |^{\pp (x)} + \grad \left[ | u (x) |^{\pp (x)}
\right] \cdot ((x - x_0) \cdot \sigma) \sigma \nonumber\\
& \eqs  \left[ 1 + \left( \partial_{\sigma} \pp (x) (x - x_0) \cdot
\sigma \right) \log | u (x) | \right] | u (x) |^{\pp (x)} \nonumber\\
&  \qquad + \pp (x) | u (x) |^{\pp (x) - 1} (\tmop{sign} u (x)) ((x -
x_0) \cdot \sigma) \partial_{\sigma} u (x) \nonumber\\
& \eqs  \left[ 1 + \left( \partial_{\sigma} \pp (x) (x - x_0) \cdot
\sigma \right) \log | u (x) | \right] | u (x) |^{\pp (x)} \nonumber\\
&  \qquad + \pp (x) | u (x) |^{\frac{\pp (x)}{\qq (x)}} (\tmop{sign} u
(x)) ((x - x_0) \cdot \sigma) \partial_{\sigma} u (x),
\label{eq:tobeintegrated}
\end{align}
where $\qq (x)$ stands for the conjugate exponent of $\pp (x)$, i.e., $1 /
\pp (x) + 1 / \qq (x) = 1$. Integrating both sides of
{\eqref{eq:tobeintegrated}} and invoking divergence theorem, we get
\begin{align}
\int_{\Om} \Big[ 1 + \left( \partial_{\sigma} \pp (x) (x - x_0) \cdot
\sigma \right) \log | u (x) | \Big] | u (x) |^{\pp (x)} \mathd x \qquad
\qquad \qquad \qquad & & \nonumber\\
\leqslant \int_{\Om} \pp (x) | u (x) |^{\frac{\qq (x)}{\pp (x)}} |
\partial_{\sigma} u (x) | | (x - x_0) \cdot \sigma | \mathd x. & &
\label{eq:temp1j}
\end{align}
By Young's inequality for products we can estimate the integrand on the
right-hand side of {\eqref{eq:temp1j}} as follows (with $\delta : \Om
\rightarrow \RR_+$ a positive measurable function to be determined
afterward)
\begin{align}
| u (x) |^{\frac{\pp (x)}{\qq (x)}} \left| \partial_{\sigma} u (x) \right| | (x -
x_0) \cdot \sigma | & \eqs  | u (x) |^{\frac{\pp (x)}{\qq (x)}} | \delta
(x) |^{\frac{1}{\qq (x)}} \left| \partial_{\sigma} u (x) \right| | \delta (x) |^{-
\frac{1}{\qq (x)}} | (x - x_0) \cdot \sigma | \nonumber\\
& \leqslant \delta (x) \frac{| u (x) |^{\pp (x)}}{\qq (x)} +
\nonumber\\
&  \qquad + \frac{| \partial_{\sigma} u (x) |^{\pp (x)}}{\pp (x)}
\left( \frac{1}{\delta (x)} \right)^{\pp (x) - 1} | (x - x_0) \cdot \sigma
|^{\pp (x)} . \nonumber
\end{align}
Therefore,
\begin{equation}
\begin{array}{c}
\pp (x) \left( | u (x) |^{\frac{\pp (x)}{\qq (x)}} \left|\partial_{\sigma} u (x)
\right| | (x - x_0) \cdot \sigma | \right) \leqslant \; \delta (x)
(\pp (x) - 1 ) | u (x) |^{\pp (x)} + \hspace{4em} \qquad
\qquad\\
\qquad \qquad \qquad \qquad \qquad \qquad \qquad \qquad + |
\partial_{\sigma} u (x) |^{\pp (x)} \left( \frac{1}{\delta (x)}
\right)^{\pp (x) - 1} | (x - x_0) \cdot \sigma |^{\pp (x)} .
\end{array} \label{eq:temp2j}
\end{equation}
From {\eqref{eq:temp1j}} and {\eqref{eq:temp2j}} it follows that
\begin{align}
\int_{\Om} \Big[ 1 - \delta (x) (\pp (x) - 1 ) + \left(
\partial_{\sigma} \pp (x) (x - x_0) \cdot \sigma \right) \log | u (x)
| \Big] \, | u (x) |^{\pp (x)} \mathd x \qquad \qquad \notag\\
\qquad \qquad \qquad \qquad \qquad \hspace{3em} \leqslant \int_{\Om} |
\partial_{\sigma} u (x) |^{\pp (x)} \left( \frac{1}{\delta (x)}
\right)^{\pp (x) - 1} | (x - x_0) \cdot \sigma |^{\pp (x)} \mathd x. \label{eq:gencase}
\end{align}
{\noindent}{\tmem{Proof of i.}} Assume the existence of $\sigma \in
\Stwo^{N - 1}$ such that $\pp (\cdot)$ is constant along $\sigma$ and that
$\pp (x) > 1$ for every $x \in \Om$. From the previous relation
{\eqref{eq:gencase}} we infer that
\begin{equation}
\int_{\Om} \left[ 1 - \delta (x) (\pp (x) - 1) \right] | u
(x) |^{\pp (x)} \mathd x  \leqslant  \int_{\Om} | \partial_{\sigma} u
(x) |^{\pp (x)} \left( \frac{1}{\delta (x)} \right)^{\pp (x) - 1} | (x -
x_0) \cdot \sigma |^{\pp (x)} \mathd x. \label{eq:gencasetemp1}
\end{equation}
Therefore, if we define $\delta (x) \assign \frac{1}{2} \frac{1}{\pp (x) -
1}$, \ then $\left[ 1 - \delta (x) (\pp (x) - 1 ) \right] =
\frac{1}{2}$ in $\Om$, and the previous inequality {\eqref{eq:gencasetemp1}}
specializes to
\begin{equation}
\int_{\Om} | u (x) |^{\pp (x)} \mathd x \leqslant \int_{\Om} |
\partial_{\sigma} u (x) |^{\pp (x)} 2^{\pp (x)} (\pp (x) - 1
)^{\pp (x) - 1} | (x - x_0) \cdot \sigma |^{\pp (x)} \mathd x.
\label{eq:gencasetemp2}
\end{equation}
Next, we note that if we remove the hypothesis that $\pp (\cdot) > 1$, i.e.,
if we assume that $\pp (\cdot)$ can also assume the value $1$, then for
every $\varepsilon > 0$, the previous inequality {\eqref{eq:gencasetemp2}}
holds with $\pp (\cdot)$ replaced by $\pp_{\varepsilon} (\cdot) \assign \pp
(\cdot) + \varepsilon$, and gives
\begin{equation}
\int_{\Om} | u (x) |^{\pp_{\varepsilon} (x)} \mathd x \leqslant
\int_{\Om} | \partial_{\sigma} u (x) |^{\pp_{\varepsilon} (x)}
2^{\pp_{\varepsilon} (x)} (\pp_{\varepsilon} (x) - 1
)^{\pp_{\varepsilon} (x) - 1} | (x - x_0) \cdot \sigma
|^{\pp_{\varepsilon} (x)} \mathd x. \label{eq:gencasetemp3}
\end{equation}
Since $u \in C^{\infty}_c (\Om)$, we can pass to the limit for
$\varepsilon \rightarrow 0^+$ to conclude that {\eqref{eq:gencasetemp2}}
holds even if $\pp (\cdot)$ is allowed to assume the value $1$ somewhere in
$\Om$, with the understanding that $(\pp (x) - 1 )^{\pp (x) - 1}
= 1$ whenever $\pp (\cdot)$ assumes the value $1$. Overall, if $\pp (\cdot)
\in C^1_b (\Om)$ is constant along the $\sigma$-direction,
$\sigma \in \Stwo^{N - 1}$, then
\begin{equation}
\int_{\Om} | u (x) |^{\pp (x)} \mathd x \leqslant \kappa (\pp) \int_{\Om} | \partial_{\sigma} u (x) |^{\pp (x)} | (x - x_0)
\cdot \sigma |^{\pp (x)} \mathd x. \label{eq:gencasetemp4}
\end{equation}
with $\kappa (\pp )$ given by {\eqref{eq:exprconstplog}}.
Finally, passing to the infimum over the point $x_0 \in \RR^{\Ndim}$ we get
{\eqref{eq:thvarLS1}}. This concludes the proof of {\tmem{i}} under the additional regularity assumption that $\pp(\cdot) \in C^1_b (\Om)$.

{\noindent}{\tmem{Proof of ii.}} By assumption, $\pp (x)
\assign \pp (| x - x_0 |)$ is radial with respect to the point $x_0 \in
\RR^{\Ndim}$, therefore
\[ \grad \pp (x) = \pp' (| x - x_0 |) \frac{x - x_0}{| x - x_0 |} . \]
Note that, to shorten notation, we identify $\pp (\cdot)$ with the
associated one-dimensional profile of $\pp (\cdot)$ formally defined by $t
\in \RR_+ \mapsto \pp (t e_N)$ with $e_N \in \Stwo^{N - 1}$. Its
derivative is here and hereafter denoted by the prime symbol. It follows
that
\[ \partial_{\sigma} \pp (x) (x - x_0) \cdot \sigma = \pp' (| x - x_0 |)
\frac{((x - x_0) \cdot \sigma)^2}{| x - x_0 |} . \]
First, we assume that $\pp (x) > 1$ for every $x \in \Om$. If the radial
profile of $\pp (\cdot)$ is decreasing and $| u | \leqslant 1$ then
\[ \left( \partial_{\sigma} \pp (x) (x - x_0) \cdot \sigma \right) (\log | u
(x) |) | u (x) |^{\pp (x)} \geqslant 0. \]
Therefore, from {\eqref{eq:gencase}}, we infer that
\begin{equation}
\int_{\Om} \left[ 1 - \delta (x) (\pp (x) - 1 ) \right] | u
(x) |^{\pp (x)} \mathd x  \leqslant  \int_{\Om} | \partial_{\sigma} u
(x) |^{\pp (x)} \left( \frac{1}{\delta (x)} \right)^{\pp (x) - 1} | (x -
x_0) \cdot \sigma |^{\pp (x)} \mathd x. \label{eq:gencasetemplog}
\end{equation}
As in the proof of {\tmem{i.}}, if we choose $\delta (x) \assign \frac{1}{2}
\frac{1}{\pp (x) - 1}$, then the previous inequality
{\eqref{eq:gencasetemplog}} specializes to
\begin{equation}
\int_{\Om} | u (x) |^{\pp (x)} \mathd x \; \leqslant \; \int_{\Om} |
\partial_{\sigma} u (x) |^{\pp (x)} 2^{\pp (x)} (\pp (x) - 1)^{\pp (x) - 1} | (x - x_0) \cdot \sigma |^{\pp (x)} \mathd x.
\label{eq:thmlog}
\end{equation}
As in the proof of {\tmem{i.}}, we can remove the hypothesis that $\pp
(\cdot) > 1$, noting that if $\pp (\cdot)$ also assumes the value $1$, then
for every $\varepsilon > 0$, the previous inequality $\pp (\cdot)$ replaced
by $\pp_{\varepsilon} (\cdot) \assign \pp (\cdot) + \varepsilon$, and a
limit process gives the validity of {\eqref{eq:thmlog}} in the general case
$\pp (\cdot) \geqslant 1$ in $\Om$.  Eventually, from {\eqref{eq:thmlog}} we get
\begin{equation}
\int_{\Om} | u (x) |^{\pp (x)} \mathd x \leqslant \kappa (\pp) \inf_{(x_0, \sigma) \in \RR^N \times \Stwo^{N - 1}} \int_{\Om} |
\partial_{\sigma} u (x) |^{\pp (x)} | (x - x_0) \cdot \sigma |^{\pp (x)}
\mathd x
\end{equation}
with $\kappa (\pp )$ given by {\eqref{eq:exprconstplog}}.
Finally, if $\Om$ is bounded in the $\sigma$-direction, then from
{\eqref{eq:thmlog}} we get
\begin{equation}
\int_{\Om} | u (x) |^{\pp (x)} \mathd x \leqslant \kappa (\pp) \inf_{x_0 \in \RR^N} \left( \sup_{x \in \Om} | (x - x_0) \cdot
\sigma |^{\pp (x)} \right) \int_{\Om} | \partial_{\sigma} u (x) |^{\pp
(x)} \mathd x.
\end{equation}
This concludes the proof of {\tmem{ii.}} under the additional regularity assumption that $\pp(\cdot) \in C^1_b (\Om)$.

{\noindent}{\tmem{Proof of i.~for $\pp(\cdot)\in L^\infty(\Om)$.}}
We want to show how the $C^1_b (\Omega)$ hypothesis of regularity on $p
(\cdot)$ can be weakened to $p (\cdot) \in L^{\infty} (\Omega)$.

We start by
showing that if $p (\cdot) \in L^{\infty} (\Omega)$ is constant in the
$\sigma$ direction for some $\sigma \in \Stwo^{N - 1}$, then
{\eqref{eq:thvarLS1}} still holds. For that, it is sufficient to prove that if
{\eqref{eq:gencasetemp2}} holds for any $p (\cdot) \in C^1_b (\Omega)$ which
is constant in the $\sigma$ direction, then it still holds for every $p
(\cdot) \in L^{\infty} (\Omega)$ which is constant along $\sigma$.

We argue as
follows. Since $u \in C^{\infty}_c (\Omega)$, the integrals on $\Omega$ in
{\eqref{eq:gencasetemp2}} can be replaced by integrals over a bounded open
subset $\mathcal{O}$ of $\Omega$ such that $\tmop{supp} u \subset \mathcal{O}
\subset \bar{\mathcal{O}} \subset \Omega$. In other words, for a given $u \in
C^{\infty}_c (\Omega)$ and for every $p_{\varepsilon} (\cdot) \in C^1_b
(\mathcal{O})$ which is constant along $\sigma$, there holds
({\tmabbr{cf.}}~{\eqref{eq:gencasetemp2}}):
\begin{equation}
  \int_{\mathcal{O}}  | u (x) |^{p_{\varepsilon} (x)} \mathd x \leqslant
  \int_{\mathcal{O}} | \partial_{\sigma} u (x) |^{p_{\varepsilon}}
  2^{p_{\varepsilon} (x)} (p (x) - 1)^{p_{\varepsilon} (x) - 1} | (x - x_0)
  \cdot \sigma |^{p_{\varepsilon} (x)} \mathd x. \label{eq:gencasetemp2new}
\end{equation}
Now, for any sufficiently small $\delta > 0$, the open $\delta$-neighborhood
of $\mathcal{O}$ defined by $\mathcal{O}_{\delta} \assign \{ x \in \Omega \of
d (x, \mathcal{O}) < \delta \}$ is still included in $\Omega$. Therefore, if
we set
\[ 
\tilde{p} (x) = \begin{cases}
p (x) & \text{if } x \in \mathcal{O}_{\delta}\ ,\\
     0 & \text{if } x \in \RR^N \setminus \mathcal{O}_{\delta} \, ,
\end{cases} 
\]
then $\tilde{p} (\cdot)$ is constant in $\mathcal{O}_{\delta}$ along $\sigma$,
and $\tilde{p} (\cdot) \in L^1 ( \RR^{\Ndim} )$.

Next, we regularize $\tilde{p} (\cdot)$. We consider a positive and symmetric
mollifier $\eta \in C^{\infty}_c \left( \RR^{\Ndim} \right)$, $\tmop{supp}
\eta \subseteq B_1$, $0 \leqslant \eta \leqslant 1$. As usual, for every
sufficiently small $\varepsilon > 0$, we set $\eta_{\varepsilon} (y) \assign
\varepsilon^{- \Ndim} \eta \left( \varepsilon^{- \Ndim} y \right)$. Also, for
every sufficiently small $\varepsilon > 0$, we define $\tilde{p}_{\varepsilon}
(x) \assign (\tilde{p} \ast \eta_{\varepsilon}) (x)$. The regularized exponent
$\tilde{p}_{\varepsilon} (\cdot)$ is in $C^1_b (\mathcal{O})$ and is constant
in the $\sigma$ direction (in $\mathcal{O}$). Indeed, for every $x \in
\mathcal{O}$ we have,
\begin{align}
  \tilde{p}_{\varepsilon} (x + (\delta / 2) \sigma) & =  \int_{\RR^{\Ndim}}
  \eta_{\varepsilon} (y - (x + (\delta / 2) \sigma))  \tilde{p} (y) \mathd y
  \nonumber\\
  & =  \int_{\RR^{\Ndim}} \eta (y)  \tilde{p} ((x + \varepsilon y) + (\delta
  / 2) \sigma) \mathd y \nonumber\\
  & \eqs  \int_{\RR^{\Ndim}} \eta (y)  \tilde{p} ((x + \varepsilon y))
  \mathd y \nonumber\\
  & \eqs  \tilde{p}_{\varepsilon} (x), \nonumber
\end{align}
at least for every $\varepsilon < (\delta / 2)$ so that $x + \varepsilon y \in
\mathcal{O}_{\delta}$ and $(x + \varepsilon y) + (\delta / 2) \sigma \in
\mathcal{O}_{\delta}$ for every $y \in B_1$. Also, $\tilde{p}_{\varepsilon}
(\cdot) \geqslant 1$ in $\mathcal{O}$ because $\tilde{p} (\cdot) \geqslant 1$
in $\mathcal{O}_{\delta}$ and $\| \eta \|_{L^1 \left( \RR^N \right)} = 1$.

After that, we know that $\tilde{p}_{\varepsilon} \rightarrow \tilde{p} (x)$
in $L^1 \left( \RR^{\NN} \right)$. In particular, up to a subsequence,
$\tilde{p}_{\varepsilon} \rightarrow p$ a.e. in $\mathcal{O}$. Since
{\eqref{eq:gencasetemp2new}} holds for every such $\tilde{p}_{\varepsilon}
(\cdot)$, by Lebesgue dominated convergence theorem, passing to the limit for
$\varepsilon \rightarrow 0$, we conclude that {\eqref{eq:gencasetemp2new}}
holds with $p_{\varepsilon} (\cdot)$ replaced by $p (\cdot)$.

{\noindent}{\tmem{Proof of ii.~for $\pp(\cdot)\in L^\infty(\Om)$.}}
We show that if
$p (\cdot) \in L^{\infty} (\Omega)$ is radial around $x_0 \in \RR^N$ then
{\eqref{eq:thmPoinclog0}} still holds.

For that, it is sufficient to prove that if {\eqref{eq:thmlog}} holds for any
$p (\cdot) \in C^1_b (\Omega)$ which is radial around $x_0 \in \RR^N$, then it
holds also for every $p (\cdot) \in L^{\infty} (\Omega)$ which is radial
around $x_0 \in \RR^N$. Note that {\eqref{eq:thmlog}} has the same expression
of {\eqref{eq:gencasetemp2}}.

As before, since $u \in C^{\infty}_c (\Omega)$, the integrals on $\Omega$ in
{\eqref{eq:thmlog}} can be replaced by integrals over a bounded open subset
$\mathcal{O}$ of $\Omega$ such that $\tmop{supp} u \subset \mathcal{O} \subset
\bar{\mathcal{O}} \subset \Omega$. In other words, for a given $u \in
C^{\infty}_c (\Omega)$ we know that for every $p_{\varepsilon} (\cdot) \in
C^1_b (\Omega)$ which is radial with respect to $x_0$ there holds:
\begin{equation}
  \int_{\mathcal{O}}  | u (x) |^{p_{\varepsilon} (x)} \mathd x \leqslant
  \int_{\mathcal{O}} | \partial_{\sigma} u (x) |^{p_{\varepsilon} (x)}
  2^{p_{\varepsilon} (x)} (p (x) - 1)^{p_{\varepsilon} (x) - 1} | (x - x_0)
  \cdot \sigma |^{p_{\varepsilon} (x)} \mathd x. \label{eq:gencasetemp2newii}
\end{equation}
Let $p (\cdot) \in L^{\infty} (\Omega)$ and denote by $\pi : \RR_+ \to
\RR_+$ the one-dimensional decreasing profile of $p (\cdot)$, i.e., $p (x) =
\pi (| x - x_0 |)$ for each $x \in \Omega$. Without loss of generality, we
assume that $\pi \geqslant 1$ is defined on $\RR_+$. We set $\mathcal{O} (x_0)
\assign \mathcal{O} \cup \{ x_0 \}$, and introduce the one-dimensional profile
$\tilde{\pi} : \RR \rightarrow \RR_+$, with compact support, obtained by
redefining $\pi (t)$ equal to $1$ if $t \in [\tmop{diam} \mathcal{O} (x_0), 2
\tmop{diam} \mathcal{O} (x_0))$,  equal to zero if $[2 \tmop{diam} \mathcal{O}
(x_0), + \infty)$, and then extending it by zero to the whole of $\RR$. We
regularize $\tilde{\pi} (\cdot) : \RR \rightarrow \RR_+$ by setting, for every
sufficiently small $\varepsilon > 0$, $\tilde{\pi}_{\varepsilon} (x) \assign
(\tilde{\pi} \ast \eta_{\varepsilon}) (x)$, where now $\eta_{\varepsilon}$ is
the one-dimensional analog of the symmetric mollifier we introduced before.
The regularized profile $\tilde{\pi}_{\varepsilon} (\cdot)$ is in $C^1_b
\left( \RR \right)$ and still decreasing on $\RR_+$. Indeed, for any $0
\leqslant t_1 \leqslant t_2$ we have
\begin{align}
  \tilde{\pi}_{\varepsilon} (t_1) & =  \int_{\RR} \eta_{\varepsilon} (s -
  t_1)  \tilde{\pi} (s) \mathd s 
  \; = \; \int_{\RR} \eta (s)  \tilde{\pi} (t_1 + \varepsilon s) \mathd s
  \nonumber\\
  & =  \int_{- t_1 / \varepsilon}^{+ \infty} \eta (s)  \tilde{\pi} (t_1 +
  \varepsilon s) \mathd s 
  \; \geqslant\;  \int_{- t_2 / \varepsilon}^{+ \infty} \eta (s)  \tilde{\pi}
  (t_2 + \varepsilon s) \mathd s \nonumber\\
  & =  \tilde{\pi}_{\varepsilon} (t_2) . \nonumber
\end{align}
After that, we know that $\tilde{\pi}_{\varepsilon} \rightarrow \tilde{\pi}
(x)$ in $L^1( \RR )$. In particular, up to a subsequence,
$\tilde{\pi}_{\varepsilon} \rightarrow \pi$ a.e. in $\mathcal{O}$. Note that,
by construction, $\tilde{\pi}_{\varepsilon} (\cdot) \geqslant 1$ in $[0,
\tmop{diam} \mathcal{O}]$ because of $\tilde{\pi} (\cdot) \geqslant 1$ in $[0,
2 \tmop{diam} \mathcal{O}]$ and $\| \eta \|_{L^1 \left( \RR \right)} = 1$.
Therefore, if we define $\tilde{p}_{\varepsilon} (x) =
\tilde{\pi}_{\varepsilon} (| x - x_0 |)$, we get that $\tilde{p}_{\varepsilon}
\rightarrow p$ a.e. in $\mathcal{O}$. Since {\eqref{eq:gencasetemp2newii}}
holds for every such $\tilde{p}_{\varepsilon} (\cdot)$, by Lebesgue dominated
convergence theorem, passing to the limit for $\varepsilon \rightarrow 0$, we
conclude that {\eqref{eq:gencasetemp2newii}} holds with $p (\cdot)$ replaced
by $\tilde{p} (\cdot)$.
\end{proof}

\begin{remark} \label{rmk:Fan}
We want to highlight the reason why our statement \tmem{ii}\@.~in Theorem~\ref{thm:VASBPI}  does not contradict the findings in \cite[Thm.~3.1]{Fan2005}. The key is in the logarithmic term in inequality {\eqref{eq:gencase}} that, when $\pp(\cdot)$ is radial with respect to $x_0$, reads under the form
\begin{equation} \label{eq:term:log}
\int_{\Om} \left[
\pp' (| x - x_0 |) \frac{((x - x_0) \cdot \sigma)^2}{| x - x_0 |} \log | u (x) |
\right]  | u (x) |^{\pp (x)} \mathd x \, .
\end{equation}
To make this term positive, either one assumes that $|u|\leqslant 1$ in $\Om$ and $\pp'(\cdot)$ negative, or that $\pp'(\cdot)$ is positive and $|u|\geqslant 1$ in $\Om$. Since we are working with functions with compact support in $\Om$ we are forced to restrict to the case $|u|\leqslant 1$.

Actually, a closer look at the proof of \tmem{ii}, shows that for an arbitrarily radial exponent $\pp (x) = \pp (| x - x_0 |)$  in $C^1_b(\Om)$ (i.e., not necessarily decreasing), the result still holds in the subset of $ C^\infty_c(\Om)$ consisting of functions that are less than or equal to $1$ where $\pp(\cdot)$ is decreasing, and greater than or equal to $1$ in the points where $\pp(\cdot)$ is increasing. In any case, in the same spirit of logarithmic Sobolev inequalities, as soon as one is not interested in having an inequality in which the lograithmic term \eqref{eq:term:log} is nonnegative, one can retain the term \eqref{eq:term:log} to obtain an inequality that holds for every radial exponent $u\in C^\infty_c(\Om)$ (i.e., regardless of any monotonicity assumption on $u$). 

The term \eqref{eq:term:log} also  explains why, for the case of increasing $\pp(\cdot)$, the construction in \cite[Thm.~3.1]{Fan2005} relies on functions that take values in the interval $[0,1]$. Indeed, this is the only possible choice if one wants to make the integrand in \eqref{eq:term:log} negative in sign so as to invalidate the Poincaré inequality.  

Finally, if one works in $C^\infty(\bar{\Om})$ then one can repeat the argument and focus on the case in which $|u|\leqslant 1$ in $\Om$. Of course, in this case, the divergence theorem produces a remainder term in the form of a surface integral. In fact, most of the results we presented still work in $C^\infty(\bar{\Om})$ and the resulting inequalities with a (surface) remainder have applications in the analysis of minimizers of energy functionals with boundary anisotropies.
\end{remark}

\section{Acknowledgment}
The authors acknowledge support from ESI, the Erwin Schr{\"o}dinger
International Institute for Mathematics and Physics in Wien, given in occasion
of the Workshop on New Trends in the Variational Modeling and Simulation of
Liquid Crystals held at ESI, in Wien, on December 2-6, 2019. The first author
acknowledges support from the Austrian Science Fund (FWF) through the special
research program Taming complexity in partial differential systems (Grant SFB
F65).

\bibliographystyle{siam} 
\bibliography{master}

\end{document}